\documentclass[12pt]{amsart}
\usepackage{amsmath,amsthm,amsfonts,amssymb,eucal}

\renewcommand {\a}{ \alpha }
\renewcommand{\b}{\beta}
\newcommand{\y}{\eta}
\newcommand{\e}{\epsilon}

\renewcommand{\d}{\delta}

\newcommand{\D}{\Delta}
\newcommand{\s}{\sigma}
\renewcommand{\l}{\lambda}

\newcommand{\p}{\partial}

\newcommand{\Om}{\Omega}

\newcommand{\R}{ \mathbb R}

\newcommand{\N}{ \mathbb N}

\newcommand{\CL}{\mathcal L}

\newcommand{\CH}{\mathcal H}

\newcommand {\gd}{\mathfrak d}

\newcommand {\ba}{\mathbf a}

\newcommand {\bm}{\mathbf m}

\newcommand {\bb}{\mathbf b}
\newcommand {\bq}{\mathbf q}

\newcommand {\bz}{\mathbf z}
\newcommand {\BA}{\mathbf A}
\newcommand {\BB}{\mathbf B}

\newcommand {\BH}{\mathbf H}
\newcommand {\BI}{\mathbf I}
\newcommand {\BJ}{\mathbf J}
\newcommand {\BK}{\mathbf K}

\newcommand {\BQ}{\mathbf Q}
\newcommand {\BP}{\mathbf P}

\newcommand {\BS}{\mathbf S}
\newcommand {\BT}{\mathbf T}

\newcommand {\BZ}{\mathbf Z}

\newcommand{\wt}{\widetilde}
\newcommand{\wh}{\widehat}

\DeclareMathOperator{\res}{\restriction}

\DeclareMathOperator{\rank}{rank}

\newtheorem{thm}{Theorem}[section]
\newtheorem{cor}[thm]{Corollary}
\newtheorem{lem}[thm]{Lemma}
\newtheorem{prop}[thm]{Proposition}

\theoremstyle{definition}

\theoremstyle{remark}

\numberwithin{equation}{section}

\newcommand{\thmref}[1]{Theorem~\ref{#1}}

\newcommand{\bsymb}{\boldsymbol}

\begin{document}
\centerline{\bf ON THE SPECTRUM OF THE DIRICHLET LAPLACIAN}
\centerline{\bf IN  A NARROW STRIP,} 
\vskip0.3cm
\centerline{by Leonid Friedlander and Michael Solomyak}

\vskip0.3cm

\section{Introduction}\label{prel}

There are several reasons why the study of the spectrum of the Laplacian
in a narrow neighborhood of an embedded graph is interesting.
The graph can be embedded into a Euclidean space or it can be embedded
into a manifold. In his pioneering work \cite{CV}, Colin de Verdi\`ere
used Riemannian metrics concentrated in a small neighborhood of a graph
to prove that for every manifold $M$ of dimension greater than two and for
every positive number $N$ there exists a Riemannian metric $g$ such that
the multiplicity of the smallest positive eigenvalue of the Laplacian
on $(M,g)$ equals $N$. Recent interest to the problem is, in particular,
motivated by possible applications to mesoscopic systems. Rubinstein
and Schatzman studied in \cite{RS} eigenvalues of the Neumann Laplacian
in a narrow strip surrounding an embedded planar graph. The strip has
constant width $\epsilon$ everywhere except neighborhoods of
vertices. Under some assumptions on the structure of the strip
near vertices, they proved that eigenvalues of the Neumann Laplacian
converge to eigenvalues of the Laplacian on the graph. Kuchment and
Zheng extended in \cite{KZ} these results to the case when the strip
width is not constant. 

The Dirichlet boundary condition turns out to be
more complicated than the Neumann condition. Eigenvalues of a domain
of width $\epsilon$ are bounded from below by $\pi^2/\epsilon^2$.
Post studied in \cite{Po} eigenvalues $\lambda_j(\epsilon)$ of the Dirichlet 
Laplacian
in a neighborhood of a planar graph that has constant width $\epsilon$
near the edges and that narrows down toward the vertices.
He proved that $\lambda_j(\epsilon)-\pi^2/\epsilon^2$ converge to
the eigenvalues of the direct sum of certain Schr\"odinger operators on the 
edges with the Dirichlet boundary conditions. We show that this result can
not be extended to neighborhoods of variable width. If the width is not 
constant then the spectrum of the Dirichlet Laplacian is basically determined
by the points where it is the widest. In the paper, we treat a simple
model case: the graph is a straight segment, and the strip is the widest
in one cross-section. In this case, we derive a two-term asymptotics
for $\lambda_j(\epsilon)$.

We will formulate now main results of the paper.
Let $h(x)>0$ be a continuous function defined on a segment $I=[-a,b]$,
where $a,b>0$. We assume that

({\it i}) $x=0$ is the only point of global maximum of $h(x)$ on
$I$;

({\it ii}) The function $h(x)$ is $C^1$ on $I\setminus\{0\}$, and in a 
neighborhood of $x=0$
it admits an expansion

\begin{equation}\label{1:1}
h(x)=\begin{cases} M-c_+x^m+O\bigl(x^{m+1}\bigr),\qquad & x>0,\\
M-c_-|x|^m+O\bigl(|x|^{m+1}\bigr),\qquad & x<0\end{cases}
\end{equation}
where $M,m,c_\pm$ are real numbers and $M, c_\pm>0,\ m\ge 1$. 

If $h(x)$ is $C^\infty$ on the whole of $I$, then necessarily $m$ is even
and $c_-=c_+$.
Another interesting case is $m=1$ (profile of a `broken line'). 
\vskip0.2cm

For a positive $\e$, let
\[\Om_\epsilon=\{(x,y): x\in I,\ 0< y<\e h(x)\}.\]
Below $\D_\e$ stands for the (positive) Dirichlet Laplacian in
$\Om_\e$ and $\l_j(\e)$ for its eigenvalues. Our main goal in this
paper is to find the asymptotics of $\l_j(\e)$ as $\e\to 0$.

In theorems \ref{1:t0} -- \ref{1:t1} below the conditions ({\it i})
and ({\it ii}) are supposed to be satisfied.

\begin{thm}\label{1:t0}
Let $\a=2(m+2)^{-1}$. Then the limits
\begin{equation}\label{1:s1}
    \mu_j=\lim_{\e\to 0}\e^{2\a}\biggl(\l_j(\e)-\frac{\pi^2}
{M^2\e^2}\biggr)
\end{equation}
 exist, and $\mu_j$ are eigenvalues of the operator on $L^2(\R)$
 given by
\begin{equation}\label{1:s2}
    \BH=-\frac{d^2}{dx^2}+q(x) ,\qquad q(x)=
\begin{cases}
2\pi^2 M^{-3} c_+ x^m,\ x>0,\\
2\pi^2 M^{-3} c_- |x|^m,\ x<0.\end{cases}
    \end{equation}
\end{thm}

Note that if $m=2$ and $c_+=c_-$, the operator $\BH$ turns into the
harmonic oscillator.

\vskip0.2cm

Our second goal is to show that the eigenvalue
convergence, described by \eqref{1:s1}, can be obtained as a
consequence of a sort of uniform convergence (i.e., convergence in
norm) of the family of operators
$\left(\D_\e-\frac{\pi^2}{M^2\e^2}\right)^{-1}$. The
usual notion of uniform convergence does not make sense here, since for
different values of $\e$ the operators act in different spaces;
one needs to interpret it in an appropriate way.

\vskip0.2cm

In $L^2(\Om_\e)$ consider the subspace $\CL_\epsilon$ that
consists of functions
\begin{equation*}
\psi(x,y)=\psi_\chi(x,y)=\chi(x)\sqrt{\frac2{\e h(x)}}\,
\sin\frac{\pi y}{\epsilon h(x)};
\end{equation*}
then
\begin{equation*}
\|\psi_\chi\|^2_{L^2(\Om_\e)}=\int_I\chi^2(x)dx.
\end{equation*}
The mapping $\chi\mapsto\psi_\chi$ is an isometric isomorphism
between $\CL_\epsilon$ and $L^2(I)$. With some abuse of notations,
we will identify operators acting in $\CL_\epsilon$ with operators acting
in $L^2(I)$. 
Obviously, $\psi_\chi\in H^{1,0}(\Om_\e)$ if $\chi\in H^{1,0}(I)$. 
A direct (though, rather lengthy) computation
shows that
\begin{equation*}
    \int_{\Om_\e}|\nabla\psi_\chi|^2dxdy=\int_I\chi'(x)^2dx+
\int_I\left(\frac{\pi^2}{\e^2 h^2(x)}+ v(x)\right)\chi^2(x)dx,
\end{equation*}
where
\begin{equation*}
v(x)=\biggl(\frac{\pi^2}3+\frac14\biggr)\frac{h'(x)^2}{h^2(x)}.
\end{equation*}
 Subtracting from $\int_{\Om_\e}|\nabla\psi_\chi|^2dxdy$
the lower bound of the resulting potential, we obtain the
quadratic form (defined on $H^{1,0}(I)$)
\begin{equation}
 \bq_\e[\chi]:=\int_I\left(\chi'(x)^2+W_\e(x)\chi^2(x)\right)dx,\label{1:4a}
\end{equation}
where
\begin{equation}
 W_\e(x)=\frac{\pi^2}{\e^2}\left(\frac1{h^{2}(x)}-\frac1{M^{2}}\right)
+v(x).\label{1:4g}
\end{equation}
Since the potential $W_\e(x)$ is non-negative, and it is positive
for non-zero values of $x$, the
quadratic form \eqref{1:4a} is positive definite in $L^2(I)$. The
self-adjoint operator on $L^2(I)$, associated with $\bq_\e$, is
given by
\begin{equation}\label{1:4c}
 \BQ_\e u=-\frac{d^2u}{dx^2}+W_\e(x)u,\qquad
 u(-a)=u(b)=0.
\end{equation}

 \vskip0.2cm

The result of theorem \ref{1:t2} below can be interpreted as
two-term asymptotics, in a certain sense, of the operator-valued function
$\D_\e$ as $\e\to 0$. In its formulation, $\BI_\e$ stands for the
identity operator on $L^2(\Om_\e)$.
\begin{thm}\label{1:t2}
There exist numbers $R_0>0$ and $\e_0>0$, depending on the
function $h$ and such that
\begin{equation}\label{1:10}
\left\|\left(\D_\e-\frac{\pi^2}{M^2\e^2}\BI_\e\right)^{-1}-
\BQ_\e^{-1}\oplus\bsymb{0}\right\| \le R_0\e^{3\a},\qquad
\forall\e\in(0,\e_0).
\end{equation}
Here $\bsymb{0}$ is the zero operator on the subspace
$\CL_\e^\perp \subset L^2(\Om_\e)$.
\end{thm}

\vskip0.2cm

 The next statement describes, in what sense the
operators $\BQ_\e$ approximate the operator $\BH$ given by
\eqref{1:s2}. Introduce the family of segments
\[I_\e=(-a\e^{-\a},b\e^{-\a}),\qquad \e>0\]
and define the isometry operator $\BJ_\e:L^2(I)\to L^2(I_\e)$
generated by the dilation $x=t\e^{\a}$. 
We identify $L^2(I_\e)$ with the subspace
\[\{ u\in L^2(\R):\, u(x)=0\ {\text{a.e. on}}\ \R\setminus I_\e\}.\]
If $\BQ_\e$ is the operator \eqref{1:4c} in $L^2(I)$, then
\begin{equation}\label{1:8x}
\wh\BQ_\e:=\e^{2\a}\BJ_\e\BQ_\e\BJ_\e^{-1}
\end{equation}
is a self-adjoint operator acting in $L^2(I_\e)$.

\begin{thm}\label{1:t1}
One has
\begin{equation}\label{1:9}
\left\|\wh\BQ_\e^{-1}\oplus \bsymb0-\BH^{-1}\right\|\to 0,
\qquad\e\to 0.
\end{equation}
were $\bsymb0$ is the zero operator on the subspace
$L^2(\R\setminus I_\e)$.
\end{thm}

\vskip0.2cm
Let us present another formulation of the latter result. It suggests an 
interpretation that seems to be more transparent. However, the formulation 
as in theorem 1.3 is more convenient for the proof.

Along with the operator $\BH$ defined in \eqref{1:s2}, let us consider the 
operator family
\[ \BH_\e=-\frac{d^2}{dx^2}+\e^{-2}q(x),\quad \e>0 \]
on $L^2(\R)$, so that in particular $\BH_1=\BH$. The substitution
$x=t\e^\a$ shows that $\e^{2\a}\BH_\e$ is an isospectral family of operators.
The result of theorem 1.3 can be rewritten as
\[ \|(\e^{2\a}\BQ_\e)^{-1}\oplus 0 -(\e^{2\a}\BH_\e)^{-1}\|\to 0,
\qquad\e\to 0.\]
This shows that the family $(\e^{2\a}\BQ_\e)^{-1}$ of operators on $L^2(I)$, 
complemented by the zero operator outside $I$, approaches an isospectral family
in the norm topology.

A similar effect, in a more complicated problem of the behavior of the 
essential spectra of certain operator families, was studied by Last and 
Simon in \cite{LS}.

\vskip0.2cm
We will show now that theorems \ref{1:t2} and \ref{1:t1}  imply theorem
\ref{1:t0}. Indeed, the non-zero eigenvalues of the operator
$\BQ_\e^{-1}\oplus\bsymb0$ are the same as those of $\BQ_\e^{-1}$.
By theorem \ref{1:t2} we have for all $j\in\N$ and $\e<\e_0$:
\begin{gather*}
\left|\left(\l_j(\e)-\frac{\pi^2}{M^2\e^2}\right)^{-1}-\l_j^{-1}(\BQ_\e)\right|
\le\left\|\left(\D_\e-\frac{\pi^2}{M^2\e^2}\BI_\e\right)^{-1}-
\BQ_\e^{-1}\oplus\bsymb{0}\right\|\\  \le R_0\e^{3\a};
\end{gather*}
 therefore
\[\left|\frac1{\e^{2\a}\left(\l_j(\e)-\frac{\pi^2}{M^2\e^2}\right)}-
\frac1{\e^{2\a}\l_j(\BQ_\e)}\right|\le R_0\e^\a.\]
 Further, $\l_j(\wh\BQ_\e)=\e^{2\a}\l_j(\BQ_\e)$, so that theorem
\ref{1:t1} implies
\[\e^{2\a}\l_j(\BQ_\e)\to\mu_j\]
which coincides with \eqref{1:s1}.

\vskip0.2cm 
Theorems 1.2 and 1.3 also allow one to make some conclusions about the
behavior of the eigenfunction of the operator $\D_\e$ as $\e\to 0$.

Let $\Psi_{j,\e}(x,y)$ and $\wt\Psi_{j,\e}(x)$ be the $j$-s normalized
eigenfunctions of the operators $\D_\e$ and $\BQ_{\e}$ respectively.
Then theorem 1.2 implies that if the signs of both eigenfunctions are
chosen appropriately, then 
\begin{equation}\label{1:d}
\int_I\left|\bigl(\BP_\e\Psi_{j,\e}\bigr)(x)-\wt\Psi_{j,\e}(x)\right|^2dx
\le C_j^2\e^{6\a}.
\end{equation}
Here $\BP_\e$ is the orthogonal projection in 
$L^2(\Om_\e)$ onto the subspace $\CL_\e$. See  \eqref{3:op}
for the explicit formula  
for $\BP_\e$.

Similarly, theorem \ref{1:t1} yields
\begin{equation}\label{1:e}
\int_I\left|\wt\Psi_{j,\e}(x)-\e^{-\a/2}X_j(x\e^{-\a})\right|^2dx\to 0.
\end{equation}
where $X_j$ is the $j$-s normalized eigenfunction of the operator $\BH$.

Grieser and Jerison proved in \cite{GJ} much stronger an estimate
for the first eigenfunction in a convex, narrow domain (in our setting,
the function $h(x)$ is concave.) Similar problems are discussed in a survey
paper \cite{N} by Nazarov.

\vskip0.2cm In the next three sections we prove
theorems \ref{1:t2} and \ref{1:t1}. In section \ref{sec5new} we explain 
the derivation of the inequalities \eqref{1:d} and \eqref{1:e}, and in the 
 last section \ref{rem} we describe possible extensions of our main results.
\newline
{\bf Acknowledgements}. The bulk of the work was done when the first author
was the Weston Visiting Professor in the Weizmann Institute of Science
in Rehovot, Israel. He thanks the Institute for its hospitality.
The work was finished when both authors visited the Isaac Newton Institue
for Mathematical Sciences in Cambridge, UK. We acknowledge the hospitality
of the Newton Institute. We are also grateful to V.~Maz'ya and
S.~Nazarov for their bibliographical advice and to A.~Sobolev for
discussions.

\section{Upper bound for $\|\BQ_\e^{-1}\|$}\label{up:bnd}
As the first step, we find an upper bound
for the quantity $\|\BQ_\e^{-1}\|$ as $\e\to 0$. Notice that theorem 
\ref{1:t1} implies $\|\BQ_\e^{-1}\|\sim \mu_1^{-1}\epsilon^{2\alpha}$,
which is stronger a result than the following lemma.

\begin{lem}\label{1:ld}
Let $h(x)$ meet the properties ({\it{i}}), ({\it{ii}}) of section
\ref{prel}. Then there exists a number $R_1>0$ such that
\begin{equation}\label{1:m}
    \|\BQ_\e^{-1}\|\le R_1\e^{2\a},\qquad\forall\e>0.
\end{equation}
\end{lem}
\begin{proof}
One has (see \eqref{1:4g})
\[\frac{\e^2 W_\e(x)}{|x|^{m}}\ge
\frac{\pi^2}{|x|^{m}}\left(\frac1{h^2(x)}-\frac1{M^2}\right).\]
The function on the
right in the last inequality is strictly positive on $I$ and continuous 
on $[-a,0]$ and on $[0,b]$ 
(it equals $2\pi^2 c_\pm M^{-3}$ at $x=0\pm$; $M$ and $c_\pm$ are
numbers from \eqref{1:1}). Hence, there exists $\s>0$
such that
\begin{equation}\label{1:n}
    W_\e(x)\ge\s\e^{-2}|x|^{m},\qquad\forall \e>0,\ x\in I.
\end{equation}
The operator $-\frac{d^2}{dx^2}+\s |x|^{m}$ on $L^2(\R)$ is
positive definite; so
\[ \int_\R(\chi'(x)^2+\s |x|^{m}\chi^2(x))dx\ge R_1^{-1}\int_\R
\chi^2(x)dx,\qquad\forall \chi\in H^1(\R)\]
for some positive number $R_1$. By scaling $x\mapsto\e^{-\a}x$ we get
\[ \int_\R(\chi'(x)^2+\s \e^{-2}|x|^{m}\chi^2(x))dx\ge R_1^{-1}\e^{-2\a}\int_\R
\chi^2(x)dx,\qquad\forall \chi\in H^1(\R).\]
 In particular, this inequality is satisfied for any function
 $\chi\in H^{1,0}(I)$, extended to the whole of $\R$ by zero.
It follows from here and \eqref{1:n} that
\begin{equation}\label{1:p}
\bq_\e[\chi]\ge R_1^{-1}\e^{-2\a}\int_I \chi^2
dx,\qquad\forall\chi\in H^{1,0}(I),
\end{equation}
which implies \eqref{1:m}.
\end{proof}

\section{Proof of \thmref{1:t2}}\label{thm2}
We will systematically use the orthogonal decomposition
\begin{equation}\label{1:10a}
L^2(\Om_\e)=\CL_\e\oplus \CL_\e^\perp,
\end{equation}
and write, for $\psi\in L^2(\Om_\e)$,
\[\psi=\psi_\chi+U,\qquad \psi_\chi= \BP_\e\psi,\ U\perp \CL_\e.\]
Here $\BP_\e$ stands for the orthogonal projection in
$L^2(\Om_\e)$ onto the subspace $\CL_\e$. This projection is given
by
\begin{equation}\label{3:op}\BP_\e\psi=\psi_\chi,\ {\text{where}}\ 
\chi(x)=\sqrt{\frac2{\e
h(x)}}\int_0^{\e h(x)}\psi(x,y)\sin\frac{\pi y}{\e h(x)}dy.\end{equation}
Note that
 \[\BP_\e H^{1,0}(\Om_\e)=H^{1,0}(I).\]

The inclusion $U\in \CL_\e^\perp$ means that
\begin{equation}\label{2:1}
\int_0^{\e h(x)}U(x,y)\sin\frac{\pi y}{\e h(x)}dy=0,\qquad
{\text{for a.a.}}\ x\in I.
\end{equation}
If $U\in H^{1,0}(\Om_\e)$, then integration by parts gives
\begin{equation}\label{2:1x}
    \int_0^{\e h(x)}U'_y(x,y)\cos\frac{\pi y}{\e h(x)}dy=0,\qquad
{\text{for a.a.}}\ x\in I.
\end{equation}
Because $U$ satisfies the Dirichlet boundary condition, \eqref{2:1}
implies
\begin{gather*}
\int_0^{\e h(x)}U^2(x,y)dy\le\frac{\e^2h^2(x)}{4\pi^2}
\int_0^{\e h(x)}U'_y(x,y)^2dy\\
\le\frac{M^2\e^2}{4\pi^2}\int_0^{\e h(x)}U'_y(x,y)^2dy
\end{gather*}
and therefore, for $U\in \CL_\e^\perp\cap H^{1,0}(\Om_\e)$
\begin{equation}\label{2:4}
\|U\|^2_{L^2(\Om_\e)}\le
\frac{M^2\e^2}{3\pi^2}\int_{\Om_\e}\left(|\nabla
U|^2-\frac{\pi^2}{M^2\e^2}|U|^2\right)dxdy.
\end{equation}

In addition, if  $U\in H^{1,0}(\Om_\e)$ then
one can differentiate \eqref{2:1} with respect to $x$ to get
\begin{equation}\label{2:2}
\int_0^{\e h(x)}U'_x(x,y)\sin\frac{\pi y}{\e h(x)}dy=
\frac{\pi}{\e}\wt h(x) \int_0^{\e h(x)}yU(x,y)\cos\frac{\pi y}{\e
h(x)}dy.
\end{equation}
Here and later, we use the notation
\[ \wt h(x)=\frac{h'(x)}{h^2(x)};\]
this function repeatedly appears in our calculations. \vskip0.2cm

For the proof of theorem \ref{1:t2} we compare the quadratic forms
of the operator
\[\BA_\e=\D_\e-\frac{\pi^2}{M^2\e^2}\BI_\e\]
appearing in \eqref{1:10}, and of its diagonal part with respect
to the decomposition \eqref{1:10a}, which is
\begin{equation}\label{2:2p}
     \BB_\e=\BQ_\e\oplus\bigl((\BI-\BP_\e)\BA_\e \res
\CL_\e^\perp\bigr).
\end{equation}
 The quadratic form of $\BB_\e$ is
(again, for $\psi=\psi_\chi+U$)
\[\bb_\e[\psi]=\bq_\e[\chi]+\int_{\Om_\e}\left(|\nabla U|^2-
\frac{\pi^2}{M^2\e^2}|U|^2\right)dxdy,\]
where $\bq_\e$ is given by \eqref{1:4a}. From \eqref{1:p} and
\eqref{2:4} we conclude that with some $C>0$
\begin{equation}\label{2:2b}
    \bb_\e[\psi]\ge C\e^{-2\a}\|\psi\|^2,\qquad\forall\psi\in
    H^{1,0}(\Om_\e).
\end{equation}

The quadratic form of $\BA_\e$ is
\[\ba_\e[\psi]=\bb_\e[\psi]+2\bm_\e[\psi]\]
where (one half of) the off-diagonal term is
\[\bm_\e[\psi]=\int_{\Om_\e}\left(\nabla\psi_\chi\cdot\nabla U -
\frac{\pi^2}{M^2\e^2}\psi_\chi U \right)dxdy=
\int_{\Om_\e}(\psi_\chi)'_xU'_x dxdy;\]
 the integral containing $(\psi_\chi)'_yU'_y$ vanishes because of 
\eqref{2:1x} and the one containing
$\psi_\chi U$ vanishes  because $\psi_\chi$ and $U$ belong to
orthogonal subspaces of $L^2(\Omega_\epsilon)$.

We will now estimate $\bm_\e[\psi]$. It is convenient to work  with the
function
\[\phi(x)=h^{-1/2}(x)\chi(x).\]
instead of $\chi$.
It is easy to see that for each $\chi\in H^1(I)$ we have
\begin{equation}\label{2:2a}
\|\phi\| \asymp\|\chi\|,\qquad
\|\phi'\|^2+\|\phi\|^2 \asymp\|\chi'\|^2+\|\chi\|^2\le C\bq_\e[\chi];
\end{equation}
the symbol $\asymp$ stands for two-sided inequality. Taking \eqref{2:2} 
into account, we find that
\begin{equation}\label{2:2z}
\bm_\e[\psi]=
\frac{\sqrt2\pi}{\e^{3/2}}\int_{\Om_\e} \wt h(x)\cos\frac{\pi
y}{\e h(x)}\left(\phi'U -\phi U'_x\right)ydxdy.
\end{equation}
The next estimate follows immediately:
\begin{gather*}
\bm^2_\e[\psi]\\  \le  C\e^{-3}
\biggl(\|U\|^2_{L^2(\Om_\e)}\int_{\Om_\e}{\phi'(x)}^2y^3dxdy
+\|U'_x\|^2_{L^2(\Om_\e)}\int_{\Om_\e}\phi^2(x)y^3dxdy\biggr)\\
\le C\left(\|U\|^2_{L^2(\Om_\e)}\|\phi'\|^2_{L^2(I)}+
\|U'_x\|^2_{L^2(\Om_\e)}\|\phi\|^2_{L^2(I)}\right).
\end{gather*}
Now we conclude from \eqref{2:4}, \eqref{2:2a}, and \eqref{1:p} that
\[ \bm^2_\e[\psi]\le C\left(\e^2\bb_\e[U]\bq_\e[\chi]+
\e^{2\a}\bb_\e[U]\bq_\e[\chi]\right),\]
whence
\begin{equation}\label{2:3}
 |\bm_\e[\psi]|\le C'\e^{\a}\bb_\e[\psi].
\end{equation}
It is important that the constant $C'$ does not depend on $\e$.

The quadratic form $\bb_\e$ is positive definite. Choosing
 $\e_0=(4C')^{-1/\a}$, we conclude from \eqref{2:3} that
\begin{equation*}
    (1-C'\e^\a)\bb_\e[\psi]\le \ba_\e[\psi]\le
    (1+C'\e^\a)\bb_\e[\psi],\qquad\forall\e<\e_0
\end{equation*}
for all $\psi\in H^{1}(\Om_\e)$. Hence, for $\e<\e_0$ the
quadratic form $\ba_\e$ is also positive definite. Taking
\eqref{2:2b} into account, we find that there exists a positive
constant $C$ such that
\begin{equation*}
\ba_\e[\psi]\ge C^{-1}\e^{-2\a}\|\psi\|^2,\ \bb_\e[\psi]\ge
C^{-1}\e^{-2\a}\|\psi\|^2,\qquad\forall\e<\e_0,
\end{equation*}
or, equivalently,
\begin{equation}\label{2:4f}
    \|\BA_\e^{-1}\|\le C\e^{2\a}, \ \|\BB_\e^{-1}\|\le
    C\e^{2\a},\qquad\forall\e<\e_0.
\end{equation}

The estimate \eqref{2:3} implies an estimate for the
bilinear form $\bm_\e[\psi_1,\psi_2]$ which corresponds to the
quadratic form $\bm_\e[\psi]$, i.e.
\[ \bm_\e[\psi_1,\psi_2]=\int_{\Om_e}(\psi_{\chi_1})'_x(U_2)'_x dxdy.\]
Namely,
\begin{equation}\label{2:5a}
|\bm_\e[\psi_1,\psi_2]|\le
C'\e^\a\left(\bb_\e[\psi_1]\bb_\e[\psi_2]\right)^{1/2},
\end{equation}
 with the same constant $C'$ as in \eqref{2:3}.
\vskip0.2cm Note that in the right-hand side of \eqref{2:5a}
each factor $\bb_\e[\psi_j]$ can be replaced by $\ba_\e[\psi_j]$;
that will result in a change of the constant $C'$,
which is not essential.

\vskip0.2cm

 We have, for any $\psi_1,\psi_2\in
H^{1,0}(\Om_\e)$:
\begin{gather*}
\left|(\BA_\e^{1/2}\psi_1,\BA_\e^{1/2}\psi_2)-
(\BB_\e^{1/2}\psi_1,\BB_\e^{1/2}\psi_2)
\right|=|\ba_\e[\psi_1,\psi_2]-\bb_\e[\psi_1,\psi_2]|\\
=2|\bm_\e[\psi_1,\psi_2]|\le
C\e^\a(\ba_\e[\psi_1]\bb_\e[\psi_2])^{1/2}.
\end{gather*}
Take here $\psi_1=\BB_\e^{-1}f,\ \psi_2=\BA_\e^{-1}g$, where
$f,g\in L^2(\Om_\e)$ are arbitrary. Then we get by \eqref{2:4f}:
\begin{gather*}
|(\BA_\e^{-1}f,g)-(\BB_\e^{-1}f,g)| \le C\e^\a\bigl((\BA_\e^{-1}g,g)
(\BB_\e^{-1}f,f)\bigr)^{1/2}\le C\e^{3\a}\|f\|\|g\|;
\end{gather*} therefore
\begin{equation}\label{2:6a}
\|\BA_\e^{-1}-\BB_\e^{-1}\|\le C\e^{3\a}.
\end{equation}

It follows from \eqref{2:2p} and \eqref{2:4} that
\[ \|\BB_\e^{-1}-(\BQ_\e^{-1}\oplus\bsymb0)\|=\|
\bigl((\BI-\BP_\e)\BA_\e\res \CL_\e^\perp\bigr)^{-1}\|\le
M^2\e^2/3\pi^2.\] Together with \eqref{2:6a}, this completes the
proof of theorem \ref{1:t2}.

\section{Proof of theorem \ref{1:t1}}\label{str}
In the course of the proof we rely upon the following statement.

\begin{prop}\label{2:pr}
$1^\circ$ Let $V(x)\ge0$ be a measurable function on $\R$, such
that $V(x)\to\infty$ as $|x|\to\infty$, and let $\{I_\e\},\
0<\e<1,$ be an expanding family of intervals:
\[I_{\e_1}\subset I_{\e_2}\ (\e_1>\e_2),\qquad \cup_{0<\e<1}I_\e=\R.\]
Consider the quadratic form
\[ \bz_V[u]=\int_\R({u'}^2+Vu^2)dx,\qquad u\in\gd_V:=
\{H^1(\R): \bz_V[u]<\infty\}\]
and a family of its restriction $\bz_{V,I_\e}$ to the domains
\[ \gd_{V,I_\e}=\{u\in\gd_V:u\left\vert_{\p I_\e}=0\right\}.\]
 Let $\BZ_V,\BZ_{V,I_\e}$ stand for the corresponding self-adjoint 
operators on $L^2(\R)$.
Then
\begin{equation}\label{1:9x}
\|\BZ_{V,I_\e}^{-1}-\BZ_V^{-1}\|\to 0,\qquad \e\to 0,
\end{equation}
\vskip0.2cm

$2^\circ$ Let a potential $V_0\ge 0$ be fixed, such that
$V_0(x)\to\infty$ as $|x|\to\infty$. Then the convergence in
\eqref{1:9x} is uniform in the class of all potentials $V$ such
that
\begin{equation*}
V(x)\ge V_0(x)\qquad{\text{on}}\ \R.
\end{equation*}
\end{prop}
\begin{proof}
$1^\circ$ Under the assumptions of proposition the strong
convergence $\BZ_{V,I_\e}^{-1}\to\BZ_V^{-1}$ is well known. For
instance, it follows from theorem VIII.1.5 in the book \cite{K}.
Its assumptions are evidently satisfied if we take into account
that $C_0^\infty(\R)$ is a core for the operator $\BZ_V$.

For each $\e$ we have $\gd_{V,I_\e}\subset\gd_V$ and
$\bz_{V,I_\e}[u]=\bz_{V}[u]$ for every $u\in \gd_{V,\e}$.
By the definition of inequalities between self-adjoit operators
(see e.g. \cite{BS}, section 10.2.3),
this means that $\BZ_{V,I_\e}\ge
\BZ_V$ and, by theorem 10.2.6 from \cite{BS},
\[\BZ^{-1}_{V,I_\e}\le \BZ^{-1}_V.\]
Since $V(x)\to\infty$ as $|x|\to\infty$, the operator $\BZ^{-1}_V$
is compact.

Now, we get the statement $1^\circ$ by applying
theorem 2.16 in \cite{S} (which is an analogue of the classical Lebesgue
theorem on dominated convergence). In particular, the theorem
says that if $\BT_\e,\ 0<\e<\e_0$ is a family of compact, self-adjoint
operators such that $\BT_\e \to \BT$ strongly, and there
exists a compact non-negative operator $\BT_0$, such that $|\BT_\e|\le
\BT_0$
for each $\e$, then $\|\BT_\e-\BT\|\to 0$.

$2^\circ$ Actually, this statement also is a consequence of
theorem  2.16 in \cite{S}. More exactly, it immediately follows from the
last displayed inequality in the proof of theorem.
\end{proof}

Each operator $\BZ_{V,I_\e}$ appearing in the formulation is the
direct sum of the operators inside and outside the interval
$I_\e$, generated by the differential expression $-d^2/dx^2+V$ and
the Dirichlet conditions at $\p I_\e$. Let us denote these
operators as $\BZ_{V,int(I_\e)},\ \BZ_{V,ext(I_\e)}$ respectively.

\begin{cor}\label{2:cor}
Both statements of proposition \ref{2:pr} remain valid if we
replace each operator $\BZ_{V,I_\e}^{-1}$ by
$\BZ_{V,int(I_\e)}^{-1}\oplus\bsymb0$, where $\bsymb0$ is the zero
operator on the subspace $\{u\in L^2(\R): u=0\ on\ \R\setminus
I_\e\}$.
\end{cor}
Indeed, this immediately follows from the fact that
\[(\BZ_{V,ext(I_\e)}u,u)\ge \|u\|^2\inf_{x\in\R\setminus
I_\e}V(x),\] whence $\|\BZ_{V,ext(I_\e)}^{-1}\|\to 0$.

\vskip0.2cm {\it Proof of theorem \ref{1:t1}}. Let $W_\e$ be the
function defined in \eqref{1:4g} and
 \[V_\e(t)=\e^{2\a}W_\e(t\e^\a).\]
Then the quadratic
form of the operator \eqref{1:8x} is
\[\wh\bq_\e[u]=\int_{I_\e}\left(u'(t)^2+V_\e(t)u^2(t)\right)dt.\]
The assumption \eqref{1:1}, say for $x>0$, can be written as
\[ h(x)=M-c_+ x^{m}+\rho(x)x^{m+1},\qquad \rho\in L^\infty(0,b).\]
Hence,
\begin{equation*}
\frac1{h^2(x)}-\frac1{M^2}=2c_+ M^{-3}x^{m}+\rho_1(x)x^{m+1},\qquad
\rho_1\in L^\infty(0,b);
\end{equation*}
therefore
\begin{equation*}
V_\e(t)=q(t)+\pi^2\rho_1(t\e^{\a})t^{m+1}\e^{\a}+\e^{2\a}v(t\e^\a),
\qquad t\in(0,b\e^{-\a}).
\end{equation*}
A similar equality is satisfied also for $t\in(-a\e^{-\a},0)$.

\vskip0.2cm
Along with $\{I_\e\}$, we need a system $\{I'_\e\}$ of narrower
intervals, say
\[ I'_\e =(-\e^{-\b},\e^{-\b}).\]
Here $\b>0$ can be taken arbitrary; the only condition is
$\b(m+1)<\a$. Then for any $\y>0$ there exists a number
$\e(\y)>0$, such that
\begin{equation}\label{2:11}
\left|V_\e(t)-q(t)\right|<\y\qquad{\text{for all}}\ t\in
I'_\e,\ \e<\e(\y).
\end{equation}
In addition, it follows from \eqref{1:n} that
\begin{equation}\label{2:12}
V_\e(t)\ge \s |t|^{m}\qquad{\text{for all}}\ t\in I_\e,\ \e<1.
\end{equation}
It is useful to extend each function $V_\e(t)$ to the whole of
$\R$, taking $V_\e(t)=\s |t|^{m}$ for $t\notin I_\e$. With each
(extended) function $V_\e$ we associate three operators: $\BH_\e$
acting on $L^2(\R)$, $ \wh\BQ_\e$ acting on
$L^2(I_\e)$, and $\wh{\BQ'}_\e$ acting on $L^2(I'_\e)$.
Each operator acts
as $-d^2/dt^2+V_\e(t)$; the last two operators
are taken with the Dirichlet boundary conditions. To apply 
proposition \ref{2:pr}, one takes
\[ \BH_\e=\BZ_{V_\e};\qquad \wh\BQ_\e=\BZ_{V_\e,int(I_\e)},\
\wh{\BQ'}_\e=\BZ_{V_\e,int(I'_\e)}.\] In particular, $\wh\BQ_\e$
is nothing but  the operator \eqref{1:8x}. 

By proposition \ref{2:pr}, $2^\circ$  we have
\[ \|{\wh \BQ_\e}^{-1}\oplus\bsymb0-\BH_\e^{-1}\|\to 0,
\qquad \|\wh{\BQ'}_\e^{-1}\oplus\bsymb0-{\BH_\e}^{-1}\|\to 0\] as
$\e\to 0$. Here $\bsymb0$ stands
for the zero operator on $L^2(\R)\ominus L^2(I_\e)$, or on
$L^2(\R)\ominus L^2(I'_\e)$. Therefore, 
\begin{equation}\label{2:12y}
\|{\wh \BQ_\e}^{-1}\oplus\bsymb0-\wh{
\BQ'}_\e^{-1}\oplus\bsymb0\|\to 0.
\end{equation}
Note that the operators $\wh\BQ_\e,\ \wh{\BQ'}_\e$ depend on the
parameter $\e$ in two ways: via the potential and via the
interval. For this reason, the statement $1^\circ$ of
proposition \ref{2:pr} is insufficient for making these conclusions.

Consider also the family of operators
$\BH'_\e:=\BZ_{q(t),int(I_\e')}$. They act on $L^2(I'_\e)$ as
\[\BH'_\e u =-u''+q(t)u,\]
with the Dirichlet conditions at $\p I'_\e$. This time, the
potential does not involve the parameter $\e$, and we conclude
from proposition \ref{2:pr}, $1^\circ$ that
\begin{equation}\label{2:12p}
\|{\BH'_\e}^{-1}\oplus\bsymb0-\BH^{-1}\|\to 0.
\end{equation}

In addition, one has
\begin{equation}\label{2:12z}
\|\wh{\BQ'}_\e^{-1}-{\BH'_\e}^{-1}\|\to 0,\qquad\e\to 0.
\end{equation}
Indeed, by Hilbert's resolvent formula,
\[\wh{\BQ'}_\e^{-1}-{\BH'_\e}^{-1}=-\wh{\BQ'}_\e^{-1}(V_\e(t)-q(t))
{\BH'_\e}^{-1}.\] Here $\| \wh{\BQ'}_\e^{-1}\|,\ \|
\wh{\BH'}_\e^{-1}\|\le C$ for all $\e<1$ (this follows from
\eqref{2:12}), and by \eqref{2:11} the norm of the multiplication
operator is smaller than an arbitrary $\y$, provided that $\e$ is
small.

Theorem \ref{1:t1} (that is, eq. \eqref{1:9}) immediately follows
from \eqref{2:12y}, \eqref{2:12p}, and \eqref{2:12z}.

\section{Eigenfunction convergence}\label{sec5new}
\subsection{}
We start from some elementary remarks from the Hilbert space theory. 
Let $e,f$ be normalized elements of a Hilbert space $\CH$, and
\begin{equation}\label{5:0}
 \BK=(\cdot,e)e-(\cdot,f)f.
\end{equation}
A direct calculation shows that the Hilbert-Schmidt norm of the operator $\BK$
is given by

\[ \|\BK\|_{HS}^2=2(1-|(e,f)|^2).\]
\vskip0.2cm

Suppose now that $(e,f)$ is real, then $\|e\pm f\|^2=2(1\pm(e,f))$
and hence,
\begin{equation}\label{5:a}
 \min(\|e- f\|,\|e+ f\|)\le\left(\|e- f\|\|e+ f\|\right)^{1/2}=
\sqrt2\, \|\BK\|_{HS}.
\end{equation}

Further, let $\CH$ be decomposed into an orthogonal sum of two
subspaces,
\[ \CH=\CH_0\oplus\CH_0^\perp\]
and let $f\in\CH_0$.  Then, along with \eqref{5:a}, one has
\begin{equation}\label{5:b}
\min(\|\BP_0e-f\|,\|\BP_0e+f\|)\le\sqrt2\, \|\BK\|_{HS} 
\end{equation}
where $\BP_0$ is the operator of orthogonal 
projection onto $\CH_0$.

\vskip0.2cm

Suppose now that $\BS,\BT$ are two self-adjoint operators in $\CH$. We assume 
that they are bounded, though this is actually not needed. 
Suppose that, on some interval $\d\in\R$, each operator has exactly one 
point of
spectrum, and this point is a simple eigenvalue. Say, $\l_0, \mu_0$ are these 
eigenvalues for $\BS,\BT$ respectively, and $e,f$ are the corresponding 
normalized  eigenvectors.
Let $\phi(\l)$ be a smooth, real-valued function on $\R$, which vanishes 
outside $\d$ and is such that 
\[ \phi(\l_0)=\phi(\mu_0)=1.\]
Then we conclude from Spectral Theorem that
\[ \phi(\BS)=(\cdot,e)e;\qquad \phi(\BT)=(\cdot,f)f,\]
 and therefore the operator $\BK$  can be represented as
\[ \BK=\phi(\BS)-\phi(\BT).\]
This representation allows us to apply the theory of double operator
integrals (see \cite{BS1}, and especially section 8 therein.) In
particular, we conclude from theorems 8.1 and 8.3 that
\[ \|\BK\|\le C\|\BS-\BT\|,\qquad C=C(\phi).\]
Since $\rank \BK\le 2$, we conclude that also
\begin{equation}\label{5:m}
\|\BK\|_{HS}\le\sqrt2\,\|\BK\|\le C\sqrt2\,\|\BS-\BT\|.
 \end{equation}
\subsection{}
Now we proceed to proving \eqref{1:d} and \eqref{1:e}. 
We start from \eqref{1:e}. 
Then we take $\CH=L^2(\R),\ \CH_0=L^2(I_\e)$. Recall that we identify 
$L^2(I_\e)$ with the subspace in $L^2(\R)$ formed by functions vanishing
outside $I_\e$. The operator $\BP_0$ acts as the restriction operator from 
$\R$ to the interval $I_\e$.

We apply the estimate \eqref{5:m} to the operators
$\BS=\BH^{-1}$ and $\BT=\wh\BQ_\e^{-1}\oplus\bsymb0$. Theorem 1.3 guarantees
that for each $j\in\N$ there exists a number $\e_j^*$, such that for any
$\e<\e_j^*$ there is a neighborhood $\d$ of the point $\mu_j^{-1}$, in which
the spectrum of $\wh\BQ_\e^{-1}$ reduces to a single and simple eigenvalue. By 
\eqref{1:8x}, this eigenvalue is $\e^{-2\a}\l_j^{-1}(\BQ_\e)$.

The eigenfunction of $\BS$ which corresponds to the eigenvalue $\mu_j^{-1}$
is $X_j(t)$, and that of $\BT$ which corresponds to the eigenvalue 
$\e^{-2\a}\l_j^{-1}(\BQ_\e)$ is equal to 
$\e^{\a/2}\wt\Psi_{j,\e}(t\e^{\a})$ on
$I$ and vanishes outside $I$. Under the appropriate choice of the sign of
$\wt\Psi_{j,\e}(x)$, we conclude from \eqref{5:m} that 
\[\int_{I_\e}\left|X_j(t)-\e^{\a/2}\wt\Psi_{j,\e}(t\e^{\a})\right|^2dt\to 0.\]
Using the substitution $t=x\e^{-\a}$, we get \eqref{1:e}. 
  
\vskip0.2cm
To get \eqref{1:d}, we apply the estimate \eqref{5:m}
to the operators 
\[\BS=\BA_\e^{-1}=(\D_\e-\frac{\pi^2}{M^2\e^2}\BI_\e)^{-1},
\qquad \BT=\BQ_\e^{-1}\oplus\bsymb0.\]
Here $\CH=L^2(\Om_\e),\ \CH_0=\CL_\e$, and $\BP_0$ is the operator $\BP_\e$
described in \eqref{3:op}. Eigenvalues of $\BS$ are 
$\l_j(\BS)=(\l_j(\e)-\frac{\pi^2}{M^2\e^2})^{-1}$, and the corresponding
eigenfunctions are $\Psi_{j,\e}(x,y)$. Recall that $\l_j(\e)$
is our notation for $\l_j(\D_\e)$. 
By theorem 1.2, for each $j\in\N$ there exists an
$\e^*_j>0$, such that for 
$\e<\e^*_j$ the point $\l_j(\BS)$ has a neighborhood $\d\subset\R$
containing exactly one eigenvalue of the operator $\BT$. This eigenvalue
is simple and necesarily coincides with $\l_j(\BQ_\e)$. The corresponding
eigenfunction is $\psi_\chi(x,y)$ with $\chi(x)=\wt\Psi_{j,\e}(x)$. 
Now, the inequality \eqref{5:m} turns into \eqref{1:d}.

 Note that 
the interval $\d$ appearing in this argument  is quite narrow for large 
values of $j$.
Indeed, its length can not exceed the number 
\[\left(\l_{j-1}(\e)-\frac{\pi^2}{M^2\e^2}\right)^{-1}-
\left(\l_{j+1}(\e)-\frac{\pi^2}{M^2\e^2}\right)^{-1}.\]
This results in a very large
constant $C=C_j$ in the corresponding inequality \eqref{1:d}.

\section{Remarks on possible extensions}\label{rem}

{\bf 1.} The result of theorem \ref{1:t0} extends to the case when $h(x)$ is 
continuous on
$I$, positive inside $I$, satisfies the condition ({\it i}),
and in a neighborhood of $x=0$ admits the expansion \eqref{1:1};
however, the function $h(x)$ is allowed to vanish at endpoints of $I$. 
A simple but 
important example of such a function is $h(t)=1-|x|$ on the segment $I=[-1,1]$.

This statement is easy to justify by using the variatonal principle in 
its simplest form. 
Namely, we construct functions $h^\pm$ in such a way that $h^+$ satisfies 
the conditions
({\it i}) and ({\it ii}) on $I$, with the same coefficients in the expansion 
\eqref{1:1},  and the inequality $h(x)\le h^+(x)$, while $h^-$ satisfies the 
conditions ({\it i}) and ({\it ii}) on a smaller segment $\wt I\subset I$, 
also with the same coefficients in
\eqref{1:1},  and the inequality $h(x)\ge h^-(x),\ x\in \wt I$. Denote
\begin{gather*}
\Om^+_\epsilon=\{(x,y): x\in I,\ 0< y<\e h^+(x)\};\\
\Om^-_\epsilon=\{(x,y): x\in \wt I,\ 0< y<\e h^-(x)\}.
\end{gather*}
Let $\l_j^{\pm}(\e)$ stand for the eigenvalues of the Dirichlet Laplacian
in $\Om^\pm_\epsilon$, then by the variatonal principle we have
\[ \l_j^+(\e)\le \l_j(\e)\le \l_j^-(\e).\]
By theorem \ref{1:t0}, the equality \eqref{1:s1} 
is valid for  $\l_j^{\pm}(\e)$; therefore 
it holds also for $\l_j(\e)$.

It remains to construct the functions $h^\pm(x)$. Let $x>0$. 
It follows from \eqref{1:1} that on some
segment $[0,\y]$ we have 
\[ |M-h(x)-c_+x^m|\le Kx^{m+1},\]
with some constant $K>0$. The function $h^+(x)$ can be obtained (for $x>0$)
as an appropriate extension of $M-c_+x^m+Kx^{m+1}$ to $[0,b]$, and
 $h^-(x)$ can be obtained as the restriction of $M-c_+x^m-Kx^{m+1}$
to a segment $[0,\wt\y]$, where $\wt\y\le\y$ is small enough, to guarantee 
$h^-(x)>0$ on $[0,\wt\y]$. For $x<0$, we construct $h^\pm(x)$ in a
similar way.
\vskip0.2cm

At the moment it is unclear to the authors, whether theorems \ref{1:t2} and 
\ref{1:t1} also
extend to the case when $h(x)$ vanishes at the endpoints of the
interval $I$. The main technical obstacle comes from the fact
that the function $\sin\frac{\pi y}{\e h(x)}$ oscillates very fast near the
points where $h(x)$ vanishes.

\vskip0.2cm
{\bf 2.} The results of all three theorems \ref{1:t0} -- \ref{1:t1} 
extend to the case when
the Dirichlet conditions at $x=-a, x=b$ are replaced by the Neumann conditions.
The argument is basically the same as for the Dirichlet problem, except for two
important points: the proofs of lemma \ref{1:ld} and of 
proposition \ref{2:pr} do not apply 
to the Neumann case.

These obstacles can be overcome, so that the results survive under
the same assumptions ({\it i}), ({\it ii}) on the function $h(x)$.
\vskip0.2cm
{\bf 3.} Theorems \ref{1:t2} and \ref{1:t1} extend to the 
case of an infinite strip, when $I$ is the whole line, or a half-line.
Let, say, $I=\R$. One has to impose additional conditions on
the behavior of $h(x)$ as $|x|\to\infty$. 
A simple condition is

({\it iii}) The function $h(x)$ is such that
\[ \limsup\limits_{|x|\to\infty}h(x)< M;\qquad 
\frac{h'}{h}\in L^\infty(\R).\]

Theorem \ref{1:t4}  below, which is an analogue of theorem \ref{1:t0}, 
looks a little bit more 
complicated than the latter, since  
the spectrum of the Dirichlet Laplacian $\D_\e$ in $\Om_\e$ is now
not necessarily discrete. In the formulation of the theorem $\nu(\e)$
stands for the bottom of of the essential spectrum of $\D_\e$, and we take
$\nu(\e)=\infty$ if the spectrum of $\D_\e$ is discrete. We denote
by $n_-(\e),\ n_-(\e)\le\infty$, the number of eigenvalues $\l_j(\e)<\nu(\e)$.

\begin{thm}\label{1:t4}
If $h(x)$ satisfies the conditions ({\it i}), ({\it ii}) and ({\it iii}),
then for $\e$ small the spectrum of $\D_\e$ below $\nu(\e)$ is non-empty
and $n_-(\e)\to\infty$ as $\e\to 0$. For each $j\in\N$ the equality 
\eqref{1:s1}
holds, where again, $\mu_j$ are eigenvalues of the operator \eqref{1:s2}.
\end{thm}

Proof of an analogue of theorem \ref{1:t1} turns out to be 
the crucial step in the analysis of the case $I=\R$. The main difficulty here
is that theorem 2.16 in \cite{S} does not apply, since the operators involved
may be non-compact. However, an appropriate substitute can be proved, and this
leads to the desired result.

A detailed exposition of the material related to remarks 2 and 3 will 
be given in a
forthcoming paper.

University of Arizona, Tucson, Arizona, USA

e-mail address: friedlan@math.arizona.edu
\vskip0.1cm
Weizmann Institute of Science, Rehovot, Israel

e-mail address: michail.solomyak@weizmann.ac.il
\end{document}